\documentclass[10pt,reqno]{amsart}

\newtheorem{THM}{Theorem}
\newtheorem{LMA}[THM]{Lemma}
\newtheorem{PROP}[THM]{Proposition}

\numberwithin{equation}{section}

\usepackage{amsmath,latexsym}
\usepackage{amssymb}
\usepackage{euscript}
\newcommand{\showon}{\begin{eqnarray*}}
\newcommand{\showoff}{\end{eqnarray*}}

\newcommand{\Sym}{\mathfrak{S}}

\newcommand{\goesto}{\rightarrow}

\newcommand{\A}{\EuScript{A}} \renewcommand{\a}{\mathbf{a}}

   \newcommand{\CC}{\mathbb{C}}

\newcommand{\HH}{\EuScript{H}}

\newcommand{\K}{\EuScript{K}}

\newcommand{\M}{\EuScript{M}}

\newcommand{\zz}{\mathbf{z}}
\newcommand{\ww}{\mathbf{w}}

\begin{document}

\title[A converse to Grace-Walsh-Szeg\H{o}]{A converse to
the Grace-Walsh-Szeg\H{o} Theorem}

\author[P.~Br\"and\'en]{Petter Br\"and\'en}\thanks{The first author was partially supported 
by the G\"oran Gustafsson Foundation and the second author was supported by 
NSERC Discovery Grant OGP0105392. }
       \address{Department of Mathematics, Royal Institute of 
Technology,
       SE-100 44 Stockholm, Sweden}
       \email{pbranden@math.kth.se}
       
\author[D.~G.~Wagner]{David G. Wagner}
       \address{Department of Combinatorics and Optimization, 
University of Waterloo,
       Ontario, Canada N2L 3G1}
       \email{dgwagner@math.waterloo.ca}
	   

\begin{abstract}
 We prove that the symmetrizer of  a permutation  group preserves stability if and only if the group is orbit homogeneous. A consequence is that the hypothesis of permutation invariance in the
Grace-Walsh-Szeg\H{o} Coincidence Theorem cannot be relaxed. In the process we obtain a new characterization of the {\em Grace-like polynomials} introduced by D. Ruelle, and  prove that the class of such polynomials can  be endowed with a natural multiplication. 
\end{abstract}

\maketitle
\section{Introduction and main result}
The Grace-Walsh-Szeg\H{o} Coincidence Theorem  is undoubtably one of 
the most useful
phenomena governing the location of zeros of multivariate complex 
polynomials -- see \cite{RS}.  It is 
also very robust -- after a full century of investigation, 
no significant
improvement on it has been made. (Although significant generalizations to other settings have been obtained \cite{H}.) The main purpose of this note is to show 
that, in one sense at least, no such improvement is possible.\\

A \emph{circular region} is a proper subset of the complex plane that 
is bounded by
either a circle or a straight line, and is either open or closed.  A 
polynomial is
\emph{multiaffine} provided that each variable occurs at most to the 
first power.
\begin{PROP}[Grace-Walsh-Szeg\H{o} \cite{G,Sz,W}]
Let $f\in\CC[z_1,\ldots,z_n]$ be a multiaffine polynomial
that is invariant under all permutations of the variables.
Let $\A\subset\CC$ be a circular region.  Assume that either $\A$ is 
convex or that
the degree of $f$ is $n$.  For any $\zeta_1,\ldots,\zeta_n\in\A$ there 
is a $\zeta\in\A$
such that
$$f(\zeta_1,...,\zeta_n)=f(\zeta,\ldots,\zeta).$$
\end{PROP}
The question we address is:\ for which permutation groups 
$G\leq\Sym_n$
does an analogue of the Grace-Walsh-Szeg\H{o} Theorem hold?  That is, 
can one
relax the hypothesis that $f$ is invariant under \emph{all} 
permutations of the variables
by requiring only that $f$ be invariant under the permutations in a 
subgroup $G$?  (As usual, the action of $\Sym_n$ on $\CC[z_1,\ldots,z_n]$
is defined for each $\sigma\in \Sym_n$ by
$\sigma(z_i)=z_{\sigma(i)}$ and algebraic extension, and this
restricts to define an action of any subgroup.)
If $G \leq \Sym_n$ then a polynomial
$f \in \CC[z_1,\ldots, z_n]$ is \emph{$G$-invariant} if 
$\sigma(f)=f$ for all $\sigma \in G$.
Let $\HH$ denote the open upper half-plane.
Our main theorem is the following: 
\begin{THM}\label{noext}
Let $G \leq \Sym_n$ be a permutation group. Suppose that for any 
multiaffine $G$-invariant polynomial 
$f \in \CC[z_1,\ldots,z_n]$ and any $\zeta_1,\ldots, \zeta_n \in \HH$
there is a $\zeta \in \HH$ such that 
  $$f(\zeta_1,\ldots,\zeta_n)=f(\zeta,\ldots,\zeta).$$
Then every $G$-invariant polynomial is also $\Sym_n$-invariant (i.e., symmetric). 
\end{THM}

\section{Stability preservers and Grace-like polynomials}
In order to prove Theorem \ref{noext} we relate the problem at 
hand to two similar problems, which we now describe.

A multivariate polynomial
$f\in\CC[z_1,\ldots,z_n]$ is \emph{stable} provided that it is 
nonvanishing on $\HH^n$. We say that a linear operator, $T$, on polynomials 
{\em preserves stability} if $T(f)$ is stable or identically zero 
whenever $f$ is stable. We are interested in whether a linear 
transformation of the form
\begin{equation}\label{tform}
T= \sum_{\sigma \in \Sym_n}c_\sigma \sigma
\end{equation}
preserves stability, in which $\{c_\sigma:\ \sigma\in\Sym_n\}$ are 
complex numbers. 
The following proposition gives a point of attack on such questions.
\begin{PROP}[Borcea-Br\"and\'en \cite{BB}]
Let $\M_n$ be the space of complex multiaffine polynomials in $n$ 
variables. For a linear transformation $T : \M_n \rightarrow \M_n$ 
whose image is at least two-dimensional the following are
equivalent:
\begin{enumerate}
\item
$T$ preserves stability;
\item
the polynomial $T\left(\prod_{i=1}^n(z_i+w_i)\right)$
in $\M_{2n}$ is stable. 
\end{enumerate}
\end{PROP}
(In part (2) of Proposition 3, $T$ acts on $\M_{2n}\simeq
\M_n \otimes \M_{n}$ \emph{via} its given action on the $\M_n$
factor supported on the $z$ variables.)

The study of stability-preserving operators of the form \eqref{tform} 
was initiated by
Ruelle \cite{R} from a slightly different point of view.  He studied 
multiaffine complex polynomials
$P(z_1,\ldots, z_m, w_1, \ldots, w_n)$ that are nonvanishing 
whenever  $z_1,\ldots, z_m$ are
separated from $w_1, \ldots, w_n$ by a circle of the Riemann sphere, 
and termed such polynomials
\emph{Grace-like}.  A reformulation of the Grace-Walsh-Szeg\H{o} 
Theorem is that the polynomial 
\begin{equation}\label{glike}
\frac 1 {n!} \sum_{\sigma \in \Sym_n}\prod_{j=1}^n (z_{\sigma(j)}-w_j)
\end{equation}
is Grace-like. Ruelle \cite{R} proved that if all the variables in a 
Grace-like polynomial
$P$ actually do occur then $m=n$ and 
$$
P(z_1,\ldots, w_n) = \sum_{\sigma \in \Sym_n}c_\sigma \prod_{j=1}^n 
(z_{\sigma(j)}-w_j)
$$
for some complex numbers $\{c_\sigma:\ \sigma \in \Sym_n\}$.

The following proposition relates Grace-like polynomials to 
stability-preserving linear transformations. 
\begin{PROP}\label{TFAE}
Let  $\{c_\sigma:\ \sigma \in \Sym_n\}$ be complex numbers. 
The following are equivalent:
\begin{enumerate}
\item
the polynomial 
$$
P(z_1,\ldots, z_n, w_1,\ldots, w_n) =   \sum_{\sigma \in \Sym_n}c_\sigma \prod_{j=1}^n 
(z_{\sigma(j)}-w_j)
$$
is Grace-like; 
\item
the linear operator defined by $T= \sum_{\sigma \in \Sym_n}c_\sigma 
\sigma$
preserves stability; 
\item
the polynomial 
$$
Q(z_1,\ldots, z_n, w_1,\ldots, w_n) = \sum_{\sigma \in \Sym_n}c_\sigma \prod_{j=1}^n 
(z_{\sigma(j)}+w_j)
$$
is stable. 
\end{enumerate}
\end{PROP}
\begin{proof}
The equivalence of (2) and (3) is a special case of Proposition 3.

Clearly (1) implies (3) since the real line separates $z_1, \ldots, 
z_n$ from $-w_1, \ldots, -w_n$ whenever $z_1, \ldots, w_n \in \HH$. 

To see that (3) implies (1), assume that $Q$ is stable and that $z_1, 
\ldots, z_n$ and 
$w_1, \ldots, w_n$ are separated by a circle of the Riemann sphere.  
By symmetry we may assume that 
$z_1, \ldots, z_n$ lie in a convex open circular domain $\A$, and  
$w_1, \ldots, w_n$ all lie in
the interior of the complement of $\A$.  Let 
$$
\phi(z) = \frac {a z +b}{cz+d}, \quad ad-bc \neq 0,
$$ 
be a M\"obius transformation that maps $\A$ to $\HH$. Then 
$$
P(z_1, \ldots, w_n)= P(\phi(z_1), \ldots, \phi(w_n))(ad-bc)^n 
\prod_{j=1}^n (cz_j+d)(cw_j+d) 
$$
(see \cite[Proposition 5]{R}). However,  
\begin{eqnarray*}
& & P(\phi(z_1), \ldots, \phi(z_n), \phi(w_1), \ldots, \phi(w_n))\\
&=& Q(\phi(z_1), \ldots, \phi(z_n), -\phi(w_1), \ldots, -\phi(w_n))
\end{eqnarray*}
and
$\phi(z_1), \ldots, -\phi(w_n) \in \HH$, so that (3) implies
that $P(z_1, \ldots, w_n) \neq 0$, which proves (1).
\end{proof}
A consequence of Proposition \ref{TFAE} is that we may view the class of Grace-like polynomials as a multiplicative sub-semigroup of the group ring $\CC[\Sym_n]$. 
Indeed, by Proposition~\ref{TFAE} we may  identify the class of Grace-like polynomials with the set of elements $\sum_{\sigma \in \Sym_n} c_\sigma \sigma \in \CC[\Sym_n]$ for which the corresponding linear operator on $\M_n$ preserves stability. Since such linear operators are closed under composition the claim follows.

\section{Proof of the main result }

A permutation group $G\leq \Sym_n$ is \emph{$k$-homogeneous}
provided that for any two subsets $A,B\subseteq\{1,\ldots,n\}$ with
$|A|=|B|=k$ there is a $\sigma\in G$
such that $\{\sigma(a):\ a\in A\}=B$.  The permutation group $G$
is \emph{homogeneous} if it is $k$-homogeneous for all $k$. The \emph{$G$-symmetrizer} of a permutation group $G$ acting on $\{1,\ldots,n\}$ is the linear transformation of the form \eqref{tform} defined by
$$T_G=\frac{1}{|G|}\sum_{\sigma\in G} \sigma.$$
Note that $T_G= T_{\Sym_n}$ if and only if $G$ is homogeneous.

A homogeneous permutation group is necessarily transitive.
More generally, denote the orbits of $G\leq \Sym_n$ acting
on $\{1,\ldots,n\}$ by $S_1$, $S_2$, \ldots, $S_r$.  Following \cite{CD} we say that
$G$ is \emph{orbit homogeneous} provided that for any two subsets
$A,B\subseteq\{1,\ldots,n\}$ with $|A\cap S_i|=|B\cap S_i|$
for each $1\leq i\leq r$ there is a $\sigma\in G$
such that $\{\sigma(a):\ a\in A\}=B$.  If $\Sym(S_i)$ denotes
the symmetric group on the set $S_{i}$, then the operators
$T_{\Sym(S_i)}$ commute, and $G$ is orbit homogeneous
if and only if $T_{G}=T_{\Sym(S_1)}\circ\cdots\circ T_{\Sym(S_r)}$.

For our main result we require the following elementary facts.
\begin{LMA}[See Section 2 of \cite{COSW}]
Assume that $f \in \CC[z_1,\ldots, z_n]$ is stable. Then 
\begin{enumerate}
\item any polynomial obtained from $f$ by fixing some of the
variables to values in the closed upper half-plane is stable or
identically zero;
\item any polynomial obtained from $f$ by setting some of the
variables equal to one another is stable. 
\end{enumerate}
\end{LMA}

\begin{THM}\label{sh}
For a permutation group $G\leq\Sym_n$ the following are 
equivalent:
\begin{enumerate}
\item
for any multiaffine stable polynomial
$f\in\M_{n}$, $T_G(f)$ is stable;
\item
the permutation group $G$ is orbit homogeneous.
\end{enumerate}
\end{THM}
\begin{proof}
Let the orbits of $G$ acting on $\{1,\ldots,n\}$ be
$S_1$, \ldots, $S_{r}$, and let $|S_i|=s_{i}$ for each
$1\leq i\leq r$.

For the easy direction, assume (2).
Since $G$ is orbit homogeneous, 
$T_G = T_{\Sym(S_1)}\circ\cdots\circ T_{\Sym(S_r)}$.
That each $T_{\Sym(S_i)}$ preserves stability follows from 
Proposition \ref{TFAE} since the polynomial \eqref{glike} is 
Grace-like.  This establishes (1).

For the converse, assume (1).  To prove that $G$ is
orbit homogeneous consider any $r$-tuple of natural numbers
$\a=(a_1,\ldots,a_{r})$ such that $a_i\leq s_i$ for
each $1\leq i\leq r$.  A subset $A\subseteq\{1,\ldots,n\}$
such that $|A\cap S_i|=a_i$ for each $1\leq i\leq r$ is said to
have \emph{profile} $\a$.  We must show that $G$ has a single
orbit in its action on subsets with profile $\a$.
We do this by induction on $|\a|=a_1+\cdots+a_r$.  The
basis of induction, the case $|\a|\leq 1$, is a restatement of
the fact that the orbits of $G$ acting on $\{1,\ldots,n\}$ are
$S_1$, \ldots, $S_r$.

The induction step falls into two cases
\begin{itemize}
\item[(i)] there is a unique index $1\leq i\leq r$ such that $a_i>0$;\ or
\item[(ii)] there are at least two indices $1\leq i<j\leq r$ such that
$a_i>0$ and $a_j>0$.
\end{itemize}
For case (i) of the induction step we can re-index the elements
of $\{1,\ldots,n\}$ and the orbits of $G$ on points so that
$a_1=k>0$ and $S_1=\{1,\ldots,m\}$.
Let $\K$ be any orbit of $G$ acting on $k$-subsets of $S_1$.
By re-indexing the variables we may assume that $\{1,\ldots,k\}\in \K$.
From the hypothesis (1) it follows that
$$F(\zz) = T_G\left((z_1+1)(z_2+1)\cdots (z_k+1)\right)$$ 
is stable.  (The notation $\zz$ is short for $(z_1,\ldots,z_n)$.)
From the induction hypothesis it follows that
$$F(\zz) = 
\frac{1}{|\K|}\sum_{S\in\K}\zz^S + \sum_{j=0}^{k-1}e_j(S_1)
\frac {\binom k j} {\binom m j},$$
in which $e_j(S_1)=\sum_{J\subseteq S_{1}:|J|=j} \zz^J$ is the $j$-th
elementary symmetric function of the variables $\{z_v:\ v\in S_1\}$ and $\zz^J= \prod_{j \in J}z_j$. 

Recall that the Newton inequalities say that if all the zeros of a polynomial $\sum_{j=0}^d \binom d j b_j t^j$ are real then $b_j^2 \geq b_{j-1}b_{j+1}$ for all $1 \leq j \leq d-1$. 
The univariate polynomial $p(t)$
obtained from $F$ by setting $z_1=\ldots=z_k=t$ and $z_j=0$ for all $j >k$ is stable by 
Lemma 5 which means that all its zeros are real. Now
$$
p(t)=\frac{1}{|\K|}t^k + \frac{k^2}{\binom m {k-1}}t^{k-1} + \frac{ k^2(k-1)^2}{4\binom m {k-2}}t^{k-2}+\cdots, 
$$
so the Newton inequalities for the top three coefficients imply 
$$
|\K| \geq \frac 1 2 \frac {m-k+2}{m-k+1} \binom m k > \frac 1 2 
\binom m k.  
$$
Thus, every orbit of $G$ acting on $k$-subsets of $S_1$ 
contains strictly more than half of the total number of
$k$-subsets of $S_1$.  
Therefore, there is only one such orbit.
This completes case (i) of the induction step.

For case (ii) of the induction step, let $A\subseteq\{1,\ldots,n\}$
be such that $|A\cap S_i|=a_i$ for each $1\leq i\leq r$.
Re-index the orbits of $G$ acting on points
so that $a_i>0$ if and only if $1\leq i\leq p$,
where $p\geq 2$ since we are in case (ii).
Let $\K$ be the orbit containing $A$ in the action of $G$ on subsets
with profile $\a$.  By the hypothesis (1) it follows that
$$
F(\zz)=T_G \left( \prod_{v\in A}(z_{v}+1) \right)
$$
is stable.  From the induction hypothesis it follows that 
$$
F(\zz)=\prod_{i=1}^p F_i(S_i)
+ \frac 1 {|\K|} \sum_{B\in\K} \zz^{B}
- \prod_{i=1}^p \frac{e_{a_{i}}(S_{i})}{\binom {s_i}{a_i}} 
$$
in which each $F_i$ is a multiaffine polynomial in the variables indexed 
by $S_i$ --  in fact 
$$
F_i(S_i) = \sum_{j=0}^{a_i}\frac {{\binom {a_i} j}}{\binom {s_i} j}
e_j(S_i). 
$$
Next, obtain $G(\zz)$ from $F(\zz)$ by specializing $z_v=0$
for all $v\not\in A$.  The result is
$$
G(\zz) =\prod_{i=1}^p F_i(S_i\cap A)
+\left(\frac{1}{|\K|} - \prod_{i=1}^{p}{\binom{s_i}{a_i}}^{-1}
\right)\zz^A.
$$
By Lemma $5$, $G(\zz)$ is stable.  Re-index the elements of
$\{1,\ldots,n\}$ so that $i\in S_{i}\cap A$ for all $1\leq i\leq p$,
and obtain $H(z_1,\ldots,z_p)$ from $G(\zz)$ by setting
$z_v=1$ for all $v\in A$ with $p<v$.  By Lemma $5$,
$H(z_1,\ldots,z_p)$ is stable.  From the form of $G(\zz)$ we see
that
$$
H(z_1,\ldots,z_p) =\prod_{i=1}^p ( b_i + c_i z_i )
+ C \left(\frac{1}{|\K|} - \prod_{i=1}^{p}{\binom{s_i}{a_i}}^{-1}
\right)z_1\cdots z_p
$$
in which the $b_i$, $c_i$, and $C$ are strictly positive reals
(and $p \geq 2$ is an integer).
The polynomial
$$
z_1\cdots z_p H(z_1^{-1},\ldots,z_p^{-1})=
\prod_{i=1}^p ( b_i z_i + c_i )
+ C \left(\frac{1}{|\K|} - \prod_{i=1}^{p}{\binom{s_i}{a_i}}^{-1}
\right)
$$
is also stable.  Upon specializing the variables so that
$z_i = b_i^{-1}(z-c_i)$, Lemma 5 implies that the resulting polynomial
is stable -- that is
\begin{equation}\label{tpol}
z^p + C \left( \frac 1 {|\K|}- \prod_{i=1}^p{\binom 
{s_i}{a_i}}^{-1} \right).
\end{equation}
If $p\geq 3$ then stability of this polynomial implies that
$$
|\K| = \prod_{i=1}^p\binom {s_i}{a_i}. 
$$
If $p=2$ then stability of \eqref{tpol} implies that
$$
|\K| \geq \prod_{i=1}^p\binom {s_i}{a_i}. 
$$
In either case there is only one orbit of $G$ acting on subsets with 
profile $\a$.  This completes the induction step, and the proof.
\end{proof}

Homogeneous permutation groups have been classified, by Chevallay 
(unpublished) and by Beaumont and Peterson \cite{BP}.  Besides
the symmetric and alternating groups there are only four
sporadic examples.  Direct products of homogeneous 
groups are orbit homogeneous, but there are others.  For example,
let $\phi:\Sym_{n}\goesto\Sym(\{n+1, n+2\})$ be the group
homomorphism with the alternating group $\mathfrak{A}_{n}$
as kernel.  If $n\geq 3$ then the map
$\sigma\mapsto \sigma\,\phi(\sigma)$
identifies $\Sym_{n}$ with an orbit homogeneous subgroup of
$\Sym_{n+2}$.  Peter Cameron (personal communication, unpublished)
has classified orbit homogeneous groups with two orbits.  The
general case is perhaps an interesting open problem.

\begin{proof}[Proof of Theorem \ref{noext}]
Suppose that $G$ is as in the statement of Theorem \ref{noext}. We 
prove that $G$ is homogeneous, from which the conclusion of Theorem 2
follows.

Suppose that $G$ is not transitive.
Then we may partition $\{1,\ldots,n\}$ into two disjoint sets 
$A,B$ so that the polynomial 
$$
\frac 1{|A|}\sum_{j \in A}z_j - \frac i{|B|}\sum_{j \in B}z_j
$$
is $G$-invariant and constitutes a counterexample to the hypothesis 
in Theorem \ref{noext}. Hence $G$ is transitive. 

We next prove that the polynomial 
$$
F(\zz, \ww)=\frac 1 {|G|}\sum_{\sigma \in G} \prod_{j=1}^n 
(z_{\sigma(j)}+w_j)
$$
is stable.  By Proposition \ref{TFAE} and Theorem \ref{sh}, this 
will complete the proof.
Suppose that  $\zz_0, \ww_0 \in \HH^n$ are such 
that $F(\zz_0, \ww_0)=0$. The polynomial $F(\zz, \ww_0)$ is 
$G$-invariant, so by hypothesis there is a $\zeta \in \HH$ such that 
$$
0=F(\zeta, ..., \zeta, \ww_0)= \prod_{j=1}^n (\zeta+w_j^0), \quad 
\ww_0=(w_1^0, \ldots,w_n^0), 
$$
which is a contradiction since $\zeta+w_j^0 \in \HH$ for all
$1 \leq j \leq n$.
\end{proof}

\bigskip

\noindent
\textbf{Acknowledgments.} 
This research stemmed from the authors' participation in the 
Programme on \emph{Combinatorics
and Statistical Mechanics} at the Isaac Newton Institute, Cambridge, 
January to June 2008.
We thank the staff and administration of the Institute and the other 
organizers of the Programme
for providing a superb environment for research.  We also thank 
Andrea Sportiello for asking the right question, and Peter Cameron
for his insights on group theory and the references \cite{BP,CD}.
\\[3ex]

\end{document}